\documentclass[11pt]{article}
\usepackage{setspace}
\usepackage{physics}
\usepackage{dcolumn} 
\usepackage{bm}     
\usepackage{braket}
\usepackage{amsmath,amssymb,amsfonts,amsthm}
\usepackage{mathtools}
\usepackage{xcolor}
\usepackage{tikz}
\usepackage{xparse}
\usepackage{hyperref}
\usepackage{cleveref}
\usepackage{listings}
\usepackage{comment}
\usepackage{dsfont}
\usepackage[T1]{fontenc}
\usepackage{url}
\usepackage{stmaryrd} 
\usepackage{graphicx}
\usepackage[normalem]{ulem}
\usepackage{nameref}
\usepackage{authblk}
\hypersetup{colorlinks=true,citecolor=blue,linkcolor=blue,filecolor=blue,urlcolor=blue}
\usepackage[
  top=1 in,
  bottom=1in,
  left=1.2in,
  right=0.75in
]{geometry}

\providecommand{\ignore}[1]{}

\newif\ifcmnt
\cmnttrue
\ifdefined\cmntsoff\cmntfalse\fi

\ifcmnt

\newtheorem{thm}{Theorem}[section]
\newtheorem{prop}[thm]{Proposition}
\newtheorem{lem}[thm]{Lemma}

\theoremstyle{definition}

\newtheorem{example}[thm]{Example}
\numberwithin{equation}{section}

\renewcommand\bra[1]{{\langle{#1}|}}
\renewcommand\ket[1]{{|{#1}\rangle}}

\numberwithin{equation}{section}

\newtheorem*{theorem*}{Theorem}

\newtheorem*{proposition*}{Proposition}

\newtheorem*{lemma*}{Lemma}

\newcommand{\bC}{\mathbb{C}}
\newcommand{\bR}{\mathbb{R}}

\newcommand{\bF}{\mathbb{F}}

\begin{document}

\title{A majorization relation for a sum of two tensor products of positive semidefinite operators}

\author{{Mohammad A. Alhejji\footnote{Correspondence: malhejji@unm.edu} \hspace{1em} Cole Kelson-Packer}}
\affil{Center for Quantum Information and Control, University of New Mexico, Albuquerque, NM 87131, USA.}

\maketitle

\begin{abstract}
We use linear programming to prove a separable version of Ky Fan's majorization relation for a sum of two operators that are each a tensor product of \(n\) positive semidefinite operators. We give an example showing that such a relation does not hold in general for sums of three or more tensor products of three or more positive semidefinite operators. 
\end{abstract}

\section{Introduction}
\label{sec: intro}

Majorization relations are immensely useful for proving inequalities. They are also rare; Robin Hood~\cite{Arnold2011} is elusive, understandably. So it is exciting when a new majorization relation is found.

We say that a vector \(x \in \bR^d\) majorizes a vector \(y \in \bR^d\) if for each \(k \in \{1,\ldots,d\} \setminus \{d\}\) the sum of the \(k\) largest coordinates of \(x\) is greater than or equal to the sum of the \(k\) largest coordinates of \(y\), and the two vectors satisfy \(\sum_{i=1}^d x_i = \sum_{i=1}^d y_i\). This relation, which we denote by \(y \preceq x\), is equivalent to the existence of probabilities \(p_1, \ldots, p_m\) and permutation matrices \(S_1, \ldots, S_m\) satisfying \(\sum_{i=1}^m p_i S_i x = y\). It is also equivalent to the statement that the value of every symmetric concave function, entropy for example, at \(x\) is less than or equal to its value at \(y\)~\cite{Bhatia1997}.

Majorization can be extended to self-adjoint operators via their eigenvalues. We denote the eigenvalues of a self-adjoint operator \(X\) on \(\bC^d\) by \(\lambda_1 (X) \geq \cdots \geq \lambda_d(X)\) and define \( \lambda(X) := (\lambda_i(X))_{i=1}^d\). For two such operators \(X\) and \(Y\), we say \(X\) majorizes \(Y\) whenever \(\lambda(Y) \preceq \lambda(X)\). In analogy with vectors in \(\bR^d\), such a relation is equivalent to the existence of probabilities \(p_1, \ldots, p_m\) and unitary operators \(U_1, \ldots, U_m\) satisfying \(\sum_{i=1}^m p_i U_i X U_i^{*} = Y\). Analogously as well, \(X\) majorizes \(Y\) if and only if the value of every unitarily invariant concave function, quantum entropy for example, at \(X\) is less than or equal to its value at \(Y\)~\cite{Ando_1989}.

Recently, there has been interest in quantum-entropy inequalities for convex mixtures of density operators (positive semidefinite operators with unit trace) that are subject to tensor-product constraints~\cite{ahmed2026singleletteronewaydistillableentanglement,Leditzky_2023,Alhejji2024}. Our main contribution is the following theorem which implies such quantum-entropy inequalities.
\begin{thm}
\label{thm: main theorem}
Let \(d_1, \ldots, d_n\) denote positive integers. Let \(A_1, \ldots, A_n\) and \(B_1, \ldots, B_n\) denote positive semidefinite operators such that, for each \(i \in \{1,\ldots,n\}\), \(A_i\) and \(B_i\) act on \(\bC^{d_i}\). Then
\begin{align}
\label{maj rel: sep ky fan}
\lambda \left( \bigotimes_{i=1}^n A_i  + \bigotimes_{i=1}^n B_i\right) \preceq \bigotimes_{i=1}^n \lambda(A_i) + \bigotimes_{i=1}^n \lambda(B_i).
\end{align}
\end{thm}

When \(n=1\), this is merely a restatement of Ky Fan's majorization relation~\cite{Fan1949} (or eigenvalue inequality~\cite{Moslehian_2012}) specialized to a sum of two positive semidefinite operators. A proof based on linear programming for the \(n=2\) case is in Ref.~\cite{Alhejji2025-ti}. A part of this proof does not apply to cases where \(n > 2\); specifically, the proof of Proposition~4.2 in Ref.~\cite{Alhejji2025-ti} does not generalize. We amend this situation here.

Taking the same approach as in Ref.~\cite{Alhejji2025-ti}, we reduce the task of proving the relation \eqref{maj rel: sep ky fan} to proving majorization relations for sums of (orthogonal) projectors \(P_{\ell_1}^{(A)} + P_{\ell_2}^{(B)}\), where \(P_{\ell_1}^{(A)}\) projects onto a direct sum of eigenspaces corresponding to the \(\ell_1\) largest eigenvalues of \(\bigotimes_{i=1}^n A_i\) and \(P_{\ell_2}^{(B)}\) projects onto a direct sum of eigenspaces corresponding to the \(\ell_2\) largest eigenvalues of \(\bigotimes_{i=1}^n B_i\). We extensively use the fact that every tensor product of positive semidefinite operators admits an orthonormal eigenbasis with a natural product order.

For \(i \in [d] : =\{1,\ldots, d\}\), we write \(\ket{e_i}\) to denote the \(i\)th element of the standard orthonormal basis of \(\bC^d\). Furthermore, for a tuple \(x \in \Pi_{i=1}^n [d_i]\), we write \(\ket{\mathbf{e}_x}\) to denote the tensor-product vector \( \bigotimes_{i=1}^n \ket{e_{x_i}}\). In general, we use bold font to indicate tensor-product vectors. In a partially ordered set \(\mathcal{T}\), a subset \(S \subseteq \mathcal{T}\) is downward closed if for all \(x \in S\) and \(y \in \mathcal{T}\), the implication \(x \geq y \implies y \in S\) holds; \(S\) is upward closed if the dual implication, \(x \leq y \implies y \in S\), holds. In an \(n\)-fold Cartesian product of partially ordered sets, an element \(x\) precedes an element \(y\) according to the product order whenever \(x_i \leq y_i\) for all \(i \in [n]\).

\begin{thm}
\label{thm: majorization for proj sum}
Let \(\{ \ket{\mathbf{a}_x} \mid x \in \Pi_{i=1}^n [d_i]\}\) and \(\{ \ket{\mathbf{b}_x} \mid x \in \Pi_{i=1}^n [d_i]\}\) be orthonormal tensor-product bases of \(\bigotimes_{i=1}^n \bC^{d_i}\). Let \(\Omega^{(A)} \subseteq \Pi_{i=1}^n [d_i]\) and \(\Omega^{(B)} \subseteq \Pi_{i=1}^n [d_i]\) be downward-closed sets according to the product order on \(\Pi_{i=1}^n [d_i]\). Then
\begin{gather}
    \lambda\left(\sum_{x \in \Omega^{(A)}} \dyad{\bm{a}_x} + \sum_{x \in \Omega^{(B)}} \dyad{\bm{b}_x} \right)\preceq     \lambda\left(\sum_{x \in \Omega^{(A)}} \dyad{\bm{e}_x} + \sum_{x \in \Omega^{(B)}} \dyad{\bm{e}_x}\right). 
\end{gather}
\end{thm}

The contents of this paper are ordered as follows. In Sec.~\ref{sec: reduction}, we show that this theorem implies Thm.~\ref{thm: main theorem}. We prove Thm.~\ref{thm: majorization for proj sum} in Sec.~\ref{sec: proof of second theorem}. We conclude the paper in Sec.~\ref{sec: conclusion} with a counterexample showing that the relation \eqref{maj rel: sep ky fan} does not generalize to three summands when \(n \geq 3\).

\section{Reduction from Thm.~\ref{thm: main theorem} to Thm.~\ref{thm: majorization for proj sum}}
\label{sec: reduction}

For a self-adjoint operator \(X\) on \(\bC^d\) and \(k \in [d]\), we denote the operator \(\sum_{i=1}^d \lambda_i(X) \dyad{e_i}\) by \( X^{\downarrow}\) and the sum \(\sum_{i=1}^k \lambda_i(X)\) by \(\sigma_k(X)\). We observe that \eqref{maj rel: sep ky fan} is equivalent to 
\begin{align}
\label{eq: alignment sum}
\lambda \left( \bigotimes_{i=1}^n A_i  + \bigotimes_{i=1}^n B_i\right) \preceq \lambda \left( \bigotimes_{i=1}^n A_i^\downarrow  + \bigotimes_{i=1}^n B_i^\downarrow\right). 
\end{align}
By linearity of the trace, the coordinates of the two sides of this relation sum to the same number. Hence, it suffices to prove
\begin{align}
\label{ineq: sep majorization}
\sigma_k \left( \bigotimes_{i=1}^n A_i  + \bigotimes_{i=1}^n B_i\right) \leq \sigma_k \left(\bigotimes_{i=1}^n A_i^{\downarrow}  + \bigotimes_{i=1}^n B_i^{\downarrow}\right)
\end{align}
for each \(k \in [\Pi_{i=1}^n d_i-1]\).

We recall the linear programming bounds in Ref.~\cite{Alhejji2025-ti}. Let \(X\) and \(Y\) be self-adjoint operators on \(\bC^d\) with orthonormal eigendecompositions \(\sum_{i=1}^d \lambda_i(X) \dyad{x_i}\) and \(\sum_{i=1}^d \lambda_i(Y) \dyad{y_i}\), respectively. The content of Theorem~3.2 in Ref.~\cite{Alhejji2025-ti} is that \(\sigma_k(X + Y)\) is bounded from above by the optimal value of a linear maximization program whose objective function is
\begin{align}
\label{eq: objective function of lin prog}
\sum_{i=1}^d w_i^{(X)} \lambda_i(X) + \sum_{i=1}^d w_i^{(Y)} \lambda_i(Y).
\end{align}
The \(w_i^{(X)}\)'s and \(w_i^{(Y)}\)'s are the variables. The feasible set of this linear program is determined by the basic constraints
\begin{align}
\label{ineq: basic constraints}
&0 \leq w_i^{(X)} , w_i^{(Y)} \leq 1, \quad \forall i \in [d], \\
&\sum_{i=1}^d w_i^{(X)} = \sum_{i=1}^d w_i^{(Y)} = k,
\end{align}
as well as \textit{alignment constraints} which account for the overlaps between subspaces corresponding to the largest eigenvalues of \(X\) and \(Y\):
\begin{align}
\label{ineq: alignment constriants}
\sum_{i=1}^{\ell_x} w_i^{(X)} + \sum_{i=1}^{\ell_y} w_i^{(Y)} \leq \sigma_k \left( \sum_{i=1}^{\ell_x} \dyad{x_i} + \sum_{i=1}^{\ell_y} \dyad{y_i}\right), \quad \forall \ell_x, \ell_y \in [d]. 
\end{align}
We refer to the quantities appearing on the right-hand side of \eqref{ineq: alignment constriants} as \textit{alignment terms}. The optimal value of this program, \(\upsilon_k(X,Y)\), is independent of the choice of eigendecompositions for \(X\) and \(Y\); see the third remark after Theorem~3.2 in Ref.~\cite{Alhejji2025-ti}.

For simultaneously diagonalizable \(X\) and \(Y\), it holds that \(\upsilon_k(X,Y) = \sigma_k(X + Y)\) for all \(k \in [d]\); see Section~3.3 in Ref.~\cite{Alhejji2025-ti}. The two operators on the right-hand side of \eqref{ineq: sep majorization} are clearly simultaneously diagonalizable. Hence, to prove \eqref{ineq: sep majorization}, it suffices to show that  
\begin{align}
\label{ineq: sep majorization upper bound}
\upsilon_k\left( \bigotimes_{i=1}^n A_i, \bigotimes_{i=1}^n B_i\right) \leq \upsilon_k \left(\bigotimes_{i=1}^n A_i^{\downarrow}, \bigotimes_{i=1}^n B_i^{\downarrow}\right) 
\end{align}
for each \(k \in [\Pi_{i=1}^n d_i-1]\). Furthermore, the fact that \(\lambda(\bigotimes_{i=1}^n A_i) = \lambda(\bigotimes_{i=1}^n A_i^{\downarrow})\) and \(\lambda(\bigotimes_{i=1}^n B_i) = \lambda(\bigotimes_{i=1}^n B_i^{\downarrow})\) implies that the two linear programs whose optimal values are the two sides of \eqref{ineq: sep majorization upper bound} have equal objective functions. It is enough then to prove that the feasible set of the program corresponding to the right-hand side of \eqref{ineq: sep majorization upper bound} contains the feasible set of the program corresponding to the left-hand side of \eqref{ineq: sep majorization upper bound}. We accomplish this by showing that alignment terms of the right-hand side are at least as large as their counterparts of the left-hand side.

For each \(i \in [n]\), let \(\{ \ket{a_{i,j}} \}_{j=1}^{d_i}\) be an orthonormal  eigenbasis for \(A_i\) and \(\{ \ket{b_{i,j}} \}_{j=1}^{d_i}\) be an orthonormal eigenbasis for \(B_i\) such that
\begin{align}
    A_i \ket{a_{i,j}} = \lambda_j(A_i) \ket{a_{i,j}} \quad \text{and} \quad   B_i \ket{b_{i,j}} = \lambda_j(B_i) \ket{b_{i,j}}, \; \forall j \in [d_i].
\end{align}
We denote for a tuple \(x \in \Pi_{i=1}^n [d_i]\) the vector \(\bigotimes_{i=1}^n \ket{a_{i,x_i}}\) by \(\ket{\mathbf{a}_x}\) and the vector \(\bigotimes_{i=1}^n \ket{b_{i,x_i}}\) by \(\ket{\mathbf{b}_x}\), so that the set \(\{ \ket{\mathbf{a}_x} \mid x \in \Pi_{i=1}^n [d_i]\}\) is an orthonormal eigenbasis for \(\bigotimes_{i=1}^n A_i\) and the set \(\{ \ket{\mathbf{b}_x} \mid x \in \Pi_{i=1}^n [d_i]\}\) is an orthonormal eigenbasis for \(\bigotimes_{i=1}^n B_i\).

Suppose that \(x,y \in \Pi_{i=1}^n [d_i]\) satisfy \(x_i \leq y_i\) for each \(i \in [n]\). Because \(A_i \geq 0\) for each \(i \in [n]\), we have the inequality \(\Pi_{i=1}^n \lambda_{x_i} (A_i) \geq  \Pi_{i=1}^n \lambda_{y_i} (A_i)\). That is, the eigenvalue of \(\bigotimes_{i=1}^n A_i\) corresponding to \(\ket{\mathbf{a}_x}\) is greater than or equal to the eigenvalue of \(\bigotimes_{i=1}^n A_i\) corresponding to \(\ket{\mathbf{a}_y}\). The same holds for the eigenvalues of \(\bigotimes_{i=1}^n B_i\) corresponding to \(x\) and \(y\). More generally, the order of the eigenvalues of a tensor product of positive semidefinite operators on \(\bigotimes_{i=1}^n \bC^{d_i}\) includes the product order on \(\Pi_{i=1}^n [d_i]\). We write \( \langle \cdot \rangle\) to denote linear span. For each \(\ell \in [\Pi_{i=1}^n d_i]\), it can be shown by induction that there exists a downward-closed set \(\Omega_{\ell}^{(A)} \subseteq \Pi_{i=1}^n [d_i]\) such that the subspace \(\langle \{ \ket{\mathbf{a}_x} \mid x \in \Omega_\ell^{(A)}\} \rangle\) is a direct sum of eigenspaces corresponding to the \(\ell\) largest eigenvalues of \(\bigotimes_{i=1}^n A_i\). By the same token, there is a downward-closed set \(\Omega_{\ell}^{(B)} \subseteq \Pi_{i=1}^n [d_i]\) such that the subspace \(\langle \{ \ket{\mathbf{b}_x} \mid x \in \Omega_{\ell}^{(B)}\} \rangle\) is a direct sum of eigenspaces corresponding to the \(\ell\) largest eigenvalues of \(\bigotimes_{i=1}^n B_i\).

The foregoing considerations show that each of the alignment terms of the left-hand side of \eqref{ineq: sep majorization upper bound} equals
\begin{align}
\sigma_k \left( \sum_{x \in \Omega^{(A)}} \dyad{\mathbf{a}_x} + \sum_{x \in \Omega^{(B)}} \dyad{\mathbf{b}_x}\right),
\end{align}
where \(\Omega^{(A)}\) and \(\Omega^{(B)}\) are downward-closed subsets of \(\Pi_{i=1}^n [d_i]\). Our goal is to show that each of these terms is less than or equal to its counterpart of the right-hand side of \eqref{ineq: sep majorization upper bound}. That is, we aim to prove that
\begin{align}
\sigma_k \left( \sum_{x \in \Omega^{(A)}} \dyad{\mathbf{a}_x} + \sum_{x \in \Omega^{(B)}} \dyad{\mathbf{b}_x}\right) \leq \sigma_k \left( \sum_{x \in \Omega^{(A)}} \dyad{\mathbf{e}_x} + \sum_{x \in \Omega^{(B)}} \dyad{\mathbf{e}_x}\right)
\end{align}
for each \(k \in [\Pi_{i=1}^n d_i -1]\). These inequalities are implied by the majorization relation
\begin{align}
    \lambda\left( \sum_{x \in \Omega^{(A)}} \dyad{\mathbf{a}_x} + \sum_{x \in \Omega^{(B)}} \dyad{\mathbf{b}_x}\right) \preceq \lambda\left(\sum_{x \in \Omega^{(A)}} \dyad{\mathbf{e}_x} + \sum_{x \in \Omega^{(B)}} \dyad{\mathbf{e}_x}\right).
\end{align}
This completes the reduction.

\section{Proof of Thm.~\ref{thm: majorization for proj sum}}
\label{sec: proof of second theorem}

The crux of our argument is that the support (orthogonal complement of the kernel) of the projector \(\sum_{x \in \Omega^{(A)}} \dyad{\bm{a}_x}\) intersects the kernel of the projector \(\sum_{x \in \Omega^{(B)}} \dyad{\bm{b}_x}\) in a subspace whose dimension is at least \(| \Omega^{(A)} \setminus \Omega^{(B)}|\), and vice-versa. This is a corollary of Prop.~\ref{prop: dimensional bound} which we prove below using two ingredients. The first ingredient is a well-known fact about pairs of complete flags (strictly-increasing chains of \(m+1\) subspaces in an \(m\)-dimensional space)~\cite{Gillespie2019}.

\begin{lem}
\label{lem: steinberg}
Let \(\bF\) be a field. For every pair of complete flags \(V_\bullet := (0 \subset V_1 \subset \cdots \subset V_m = \bF^m)\) and \(W_\bullet := (0 \subset W_1 \subset \cdots \subset W_m = \bF^m)\), there exist a basis \(\{f_1, \ldots, f_m\}\) of \(\bF^m\) and a permutation \(\rho: [m] \rightarrow [m]\) such that for each \(i \in [m]\), \(V_i = \langle \{ f_1, \ldots, f_i\}\rangle\)  and \(W_i = \langle \{ f_{\rho(1)},\ldots, f_{\rho(i)}\}\rangle\). 
\end{lem}

A proof of this fact is in Ref.~\cite{steinberg1951geometric} and a modern account of it, though restricted to real vector spaces, is in Ref.~\cite{glaudo2025simultaneousgeneratingsetsflags}. The second ingredient we use is a statement concerning the intersection of a downward-closed set and a locally-permuted upward-closed set in a product of finite chains.

\begin{prop}
\label{prop: filter and ideal intersection}
Let \(\Omega\) denote a downward-closed subset of \(\Pi_{i=1}^n [d_i]\) and \(\Upsilon\) denote an upward-closed subset of \(\Pi_{i=1}^n [d_i]\). For each \(i \in [n]\), let \(\rho_i\) denote a permutation of the elements of \([d_i]\). Then
\begin{equation}
\label{eq: filter and ideal intersection bound}
| \Omega \cap \Pi_{i=1}^n \rho_i (\Upsilon)| \geq |\Omega \cap \Upsilon|.
\end{equation}
\end{prop}

\begin{proof}
We induct on \(n\). The base case is the \(n=1\) case. We observe that \([d_1]\) is a chain, and so the cardinality of \(\Omega \cap \rho_1(\Upsilon)\) equals the number of elements in \(\rho_1(\Upsilon)\) that are less than or equal to \(|\Omega|\). If \(|\Omega| + |\Upsilon| < d_1\), then \(|\Omega \cap \Upsilon| = 0\) which is no larger than \(|\Omega \cap \rho_1(\Upsilon)|\). Else, every \(x \in [d_1]\) would be either among the first \(|\Omega|\) elements or among the last \(|\Upsilon|\) elements, which implies \(|\Omega \cup \Upsilon| = d_1\). Then, by inclusion-exclusion, \(| \Omega \cap \Upsilon| = |\Omega| + |\Upsilon| - d_1 = |\Omega| + |\rho_1(\Upsilon)| - d_1 \leq |\Omega| + |\rho_1(\Upsilon)| - | \Omega \cup \rho_1 (\Upsilon)| = | \Omega \cap \rho_1(\Upsilon)|\). This completes the base case.

Since \(\Upsilon\) is upward closed, it is a disjoint union \(\bigcup_{j=1}^{d_n} \Upsilon_j \times \{j\}\) where the \(\Upsilon_j\)'s are upward-closed subsets of \(\Pi_{j=1}^{n-1} [d_j]\) that satisfy \(\Upsilon_1 \subseteq \cdots \subseteq \Upsilon_{d_n}\). By the same token, \(\Omega\) is a disjoint union \(\bigcup_{j=1}^{d_n} \Omega_j \times \{j\}\) where the \(\Omega_j\)'s are downward-closed subsets of \(\Pi_{j=1}^{n-1} [d_j]\) that satisfy \(\Omega_{d_n} \subseteq \cdots \subseteq \Omega_1\). We observe that 
\begin{gather}
|\Omega \cap \Upsilon| = \sum_{j=1}^{d_n} | \Omega_j \cap \Upsilon_j|
\end{gather}
and that 
\begin{gather}
| \Omega \cap \Pi_{i=1}^n \rho_i (\Upsilon)|  = \sum_{j=1}^{d_n} | \Omega_{\rho_n(j)} \cap \Pi_{i=1}^{n-1} \rho_i (\Upsilon_j)|.
\end{gather}
By the inductive hypothesis,  we have the inequality \(| \Omega_{\rho_n(j)} \cap \Pi_{i=1}^{n-1} \rho_i (\Upsilon_j)| \geq | \Omega_{\rho_n (j)} \cap \Upsilon_j|\) for each \(j \in [d_n]\). It remains to prove that
\begin{gather}
\label{eq: intersection sum inequality}
\sum_{j=1}^{d_n} | \Omega_{\rho_n(j)} \cap \Upsilon_j| \geq \sum_{j=1}^{d_n} | \Omega_{j} \cap \Upsilon_j|.
\end{gather}
If \(\rho_n(1) \neq 1\), then the inclusions \(\Omega_{\rho_n(1)} \subseteq \Omega_1\) and \(\Upsilon_1 \subseteq \Upsilon_{\rho_n^{-1} (1)}\) imply
\begin{align}
| \Omega_{\rho_n(1)} \cap \Upsilon_1 | + | \Omega_1 \cap \Upsilon_{\rho_n^{-1} (1)}| \geq | \Omega_{\rho_n(1)} \cap \Upsilon_{\rho_n^{-1} (1)}| + | \Omega_1 \cap \Upsilon_1|. 
\end{align}
Therefore, the sum \(\sum_{j=1}^{d_n} | \Omega_{\rho_n(j)} \cap \Upsilon_j|\) is lower bounded by \(\sum_{j=1}^{d_n} | \Omega_{\rho'_n(j)} \cap \Upsilon_j|\), where \(\rho'_n\) is the permutation given by \(\rho'_n(1) := 1, \rho_n'(\rho^{-1}_n(1)) := \rho_n(1)\), and \(\rho'_n(x) := \rho_n(x)\) for all \(x \notin \{1, \rho_n^{-1}(1)\}\). If \(d_n \geq 2\) and \(\rho'_n(2) \neq 2\), then we apply the same argument to \(\rho_n'\) with \(2\) taking the role of \(1\) in order to transform \(\rho_n'\) into a permutation that fixes both \(1\) and \(2\). We go on in this way until \(\rho_n\) is transformed into the identity permutation on \([d_n]\). Because the value of the sum at each step either stays the same or decreases, the inequality \eqref{eq: intersection sum inequality} is proven. \end{proof}

We now prove the lower bounds on the dimensions of subspace intersections that we seek. 
\begin{prop}
\label{prop: dimensional bound}
For a field \(\bF\), let \(\{ \ket{\bm{\alpha}_x} \mid x \in \Pi_{i=1}^n [d_i]\}\) and \(\{ \ket{\bm{\beta}_x} \mid x \in \Pi_{i=1}^n [d_i]\}\) be tensor-product bases for \(\bigotimes_{i=1}^n \bF^{d_i}\). For each downward-closed subset \(\Omega \subseteq \Pi_{i=1}^n [d_i]\) and each upward-closed subset \(\Upsilon \subseteq \Pi_{i=1}^n [d_i]\), it holds that
\begin{align}
\label{eq: dimension bound}
\dim \Big( \langle \{ \ket{\bm{\alpha}_x} \mid x \in \Omega\} \rangle \cap \langle \{ \ket{\bm{\beta}_x} \mid x \in \Upsilon\} \rangle\Big) \geq |\Omega \cap  \Upsilon|.
\end{align}
\end{prop}
\begin{proof}
For each \(i \in [n]\), Lem.~\ref{lem: steinberg} guarantees the existence of a basis \(\{\ket{\gamma_{i, j}}\}_{j=1}^{d_i}\) of \(\bF^{d_i}\) and a permutation \(\rho_i : [d_i] \rightarrow [d_i]\) such that
\begin{align}
\label{eq: steinberg gen}
    \langle \{ \ket{\alpha_{i,1}}, \ldots, \ket{\alpha_{i,j}}\}\rangle &= \langle \{\ket {\gamma_{i,1}}, \ldots, \ket{\gamma_{i,j}}\}\rangle, \\ 
\label{eq: steinberg gen 2}
    \langle \{ \ket{\beta_{i, d_i}}, \ldots, \ket{\beta_{i,d_i-j+1}}\}\rangle  &= \langle \{\ket{\gamma_{i,\rho_i(1)}}, \ldots, \ket{\gamma_{i,\rho_i(j)}}\} \rangle,
\end{align}
for each \(j \in [d_i]\). Notice that the \(\beta\) basis is in reverse order. It is convenient to express \eqref{eq: steinberg gen 2} using the mirrored permutation \(\tilde{\rho}_i\) defined by \(\tilde{\rho}_i(d_i - j +1) := \rho_i(j)\) for all \(j \in [d_i]\). Rewritten, \eqref{eq: steinberg gen 2} reads 
\begin{align}
\label{eq: steinberg gen rewr}
\langle \{ \ket{\beta_{i, d_i}}, \ldots, \ket{\beta_{i,d_i-j+1}}\}\rangle  &= \langle \{\ket{\gamma_{i,\tilde{\rho}_i(d_i)}}, \ldots, \ket{\gamma_{i,\tilde{\rho}_i(d_i - j + 1)}}\} \rangle.
\end{align}

For \(x \in \Pi_{i=1}^n [d_i]\), we write \(\ket{\bm{\gamma}_{x}}\) to denote \(\bigotimes_{i=1}^n  \ket{\gamma_{i,x_i}}\). Consider an element \(y \in \Omega\). By \eqref{eq: steinberg gen}, for each \(i \in [n]\), we have the membership \(\ket{\alpha_{i, y_i}} \in \langle \{\ket {\gamma_{i,1}}, \ldots, \ket{\gamma_{i,y_i}}\}\rangle\) which implies that 
\begin{gather}
\label{eq: alpha member}
\ket{\bm{\alpha}_y} \in \bigotimes_{i=1}^n \langle \{ \ket {\gamma_{i,1}}, \ldots, \ket{\gamma_{i,y_i}} \} \rangle = \langle \{ \ket{\bm{\gamma}_x} \mid x \leq y \} \rangle.
\end{gather}
Since \(\Omega\) is downward closed, this implies \(\ket{\bm{\alpha}_y} \in \langle \{ \ket{\bm{\gamma}_x} \mid x \in \Omega\} \rangle\). A similar argument shows that \(\ket{\bm{\gamma}_y} \in \langle \{ \ket{\bm{\alpha}_x} \mid x \in \Omega\} \rangle\). Hence, we have the equality 
\begin{gather}
\label{eq: omega equality between alpha and gamma}
\langle \{ \ket{\bm{\alpha}_x} \mid x \in \Omega\} \rangle = \langle \{ \ket{\bm{\gamma}_x} \mid x \in \Omega\} \rangle.     
\end{gather}
Now, consider an element \(z \in \Upsilon\). By \eqref{eq: steinberg gen rewr}, for each \(i \in [n]\), we have the membership \(\ket{\beta_{i,z_i}} \in \langle \{ \ket{\gamma_{i,\tilde{\rho}_i(d_i)}}, \ldots, \ket{\gamma_{i, \tilde{\rho}_i (z_i)}} \} \rangle\)
which implies that 
\begin{gather}
\label{eq: beta member}
\ket{\bm{\beta}_z} \in \bigotimes_{i=1}^n \langle \{ \ket{\gamma_{i,\tilde{\rho}_i(d_i)}}, \ldots, \ket{\gamma_{i, \tilde{\rho}_i (z_i)}} \} \rangle = \langle \{ \ket{\bm{\gamma}_x} \mid z \leq x', \Pi_{i=1}^n \tilde{\rho}_i (x') = x\} \rangle.
\end{gather}
Since \(\Upsilon\) is upward closed, this implies \(\ket{\bm{\beta}_z} \in \langle \{ \ket{\bm{\gamma}_x} \mid x \in \Pi_{i=1}^n \tilde{\rho}_i (\Upsilon) \} \rangle\). 
Similarly, if \(z' \in \Pi_{i=1}^n \tilde{\rho}_i(\Upsilon)\), then \(\ket{\bm{\gamma}_{z'}} \in \langle \{ \ket{\bm{\beta}_x} \mid x \in \Upsilon\}\rangle\). Hence, we also have the equality
\begin{align}
\label{eq: upsilon equality for omega and beta}
\langle \{\ket{\bm{\beta}_x} \mid x \in \Upsilon\} \rangle = \langle \{\ket{\bm{\gamma}_x} \mid x \in \Pi_{i=1}^n \tilde{\rho}_i (\Upsilon)\}\rangle.
\end{align}

Therefore, the dimension of the intersection \(\langle \{ \ket{\bm{\alpha}_x} \mid x \in \Omega\} \rangle \cap \langle \{ \ket{\bm{\beta}_x} \mid x \in \Upsilon\} \rangle\) equals the dimension of the intersection \(\langle \{ \ket{\bm{\gamma}_x} \mid x \in \Omega\} \rangle \cap \langle \{ \ket{\bm{\gamma}_x} \mid x \in \Pi_{i=1}^n \tilde{\rho}_i (\Upsilon)\} \rangle\) which, because the \(\ket{\bm{\gamma}_x}\)'s are linearly independent, equals \(|\Omega \cap \Pi_{i=1}^n \tilde{\rho}_i (\Upsilon)|\) which, by Prop.~\ref{prop: filter and ideal intersection}, is bounded below by \(| \Omega \cap \Upsilon|\). \end{proof}

We use the following lemma to metabolize the inequality \eqref{eq: dimension bound} into a majorization relation. 

\begin{lem}
\label{lem: two proj maj rel with orth const}
Let \(U\) and \(W\) denote finite-dimensional subspaces of some complex or real inner product space. Assume that there exist two subspaces \(U_1\) and \(W_1\) satisfying \(U_1 \subseteq U\), \(W_1 \subseteq W\), \(U_1 \perp W\) and \(W_1 \perp U\). Then the projector onto \(U\), denoted with \(P_U\), and the projector onto \(W\), denoted with \(P_W\), satisfy
\begin{align}
\label{eq: maj two proj with orth compo}
\lambda( P_U + P_W) \preceq (2, \ldots, 2, 1, \ldots,1, 0, \ldots,0),
\end{align}
where the multiplicity of \(2\) is \(\min ( \dim(U) - \dim(U_1), \dim(W) - \dim(W_1))\) and the multiplicity of \(1\) is \(\dim(U) + \dim(W) - 2 \min ( \dim(U) - \dim(U_1), \dim(W) - \dim(W_1))\).
\end{lem}

\begin{proof}
Let \(U'\) denote the orthogonal complement of \(U_1\) in \(U\) and \(W'\) denote the orthogonal complement of \(W_1\) in \(W\). The sum \(P_U +P_W\) may be written as an orthogonal direct sum \((P_{U'} + P_{W'}) \oplus P_{U_1} \oplus P_{W_1}\). We apply Ky Fan's majorization relation to the first summand to obtain \(\lambda(P_{U'} + P_{W'}) \preceq \lambda(P_{U'}) + \lambda(P_{W'})\). Direct sums preserve the majorization order; see Ch.~5 in Ref.~\cite{Marshall2011}. Thus,
\begin{align}
\label{eq: estimate}
\lambda (P_U + P_W) &= \lambda(P_{U'} + P_{W'}) \oplus \lambda(P_{U_1}) \oplus \lambda(P_{W_1})  \\
&\preceq (\lambda(P_{U'}) + \lambda(P_{W'}) )\oplus \lambda(P_{U_1}) \oplus \lambda(P_{W_1}).
\end{align}
The coordinates of the right-hand side of this majorization relation are in \(\{0,1,2\}\) and they sum to \(\dim(U)+\dim(W)\). The number of coordinates equal to \(2\) is \(\min(\dim(U'), \dim(W')) = \min (\dim(U) - \dim(U_1), \dim(W) - \dim(W_1))\) and so the number of coordinates equal to \(1\) is what remains of the total: \(\dim(U) + \dim(W) - 2 \min (\dim(U) - \dim(U_1), \dim(W) - \dim(W_1))\). \end{proof}

Finally, we are ready to prove Thm.~\ref{thm: majorization for proj sum}.

\begin{proof}[Proof of Thm.~\ref{thm: majorization for proj sum}]

Since \(\Omega^{(B)}\) is downward closed, its complement \((\Omega^{(B)})^c\) is upward closed. By Prop.~\ref{prop: dimensional bound}, the dimension of \(\langle \{ \ket{\bm{a}_x} \mid x \in \Omega^{(A)}\}\rangle \cap \langle \{ \ket{\bm{b}_x} \mid x \in (\Omega^{(B)})^c\}\rangle\) is at least \(| \Omega^{(A)} \cap (\Omega^{(B)})^c| = | \Omega^{(A)} \setminus \Omega^{(B)}|\). By the same token, Prop.~\ref{prop: dimensional bound} implies that the dimension of \(\langle \{ \ket{\bm{a}_x} \mid x \in (\Omega^{(A)})^c\}\rangle \cap \langle \{ \ket{\bm{b}_x} \mid x \in \Omega^{(B)}\}\rangle\) is at least \(| \Omega^{(B)} \setminus \Omega^{(A)}|\). 

Now, we apply Lem.~\ref{lem: two proj maj rel with orth const} with \(\langle \{ \ket{\bm{a}_x} \mid x \in \Omega^{(A)}\}\rangle\) taking the role of \(U\), \(\langle \{ \ket{\bm{b}_x} \mid x \in \Omega^{(B)}\}\rangle\) taking the role of \(W\), \(\langle \{ \ket{\bm{a}_x} \mid x \in \Omega^{(A)}\}\rangle \cap \langle \{ \ket{\bm{b}_x} \mid x \in (\Omega^{(B)})^c\}\rangle\) taking the role of \(U_1\), and \(\langle \{ \ket{\bm{a}_x} \mid x \in (\Omega^{(A)})^c\}\rangle \cap \langle \{ \ket{\bm{b}_x} \mid x \in \Omega^{(B)}\}\rangle\) taking the role of \(W_1\). The claim follows by noticing that the eigenvalues of \(\sum_{x \in \Omega^{(A)}} \dyad{\bm{e}_x} + \sum_{x \in \Omega^{(B)}} \dyad{\bm{e}_x}\) are in \(\{0,1,2\}\), that the multiplicity of \(2\) is \(|\Omega^{(A)} \cap \Omega^{(B)}|\) (its largest possible value), and that the multiplicity of \(1\) is \(|\Omega^{(A)} \setminus \Omega^{(B)}| + |\Omega^{(B)} \setminus\Omega^{(A)}|\). \end{proof}

\section{Concluding remarks}
\label{sec: conclusion}
One might hope that the relation \eqref{maj rel: sep ky fan} generalizes to every number of summands. That is,
\begin{align}
\label{eq: maj sus}
\lambda \left( \sum\limits_{j=1}^m\bigotimes_{i=1}^n A^{(j)}_i\right) \stackrel{\mathclap{\mathbf{?}}}{\preceq} \sum\limits_{j=1}^m\bigotimes_{i=1}^n \lambda(A^{(j)}_i),
\end{align}
whenever \(A^{(j)}_i \geq 0\) for all \(i \in [n]\) and \(j \in [m]\). However, this is not true for every number of summands $m$ and every number of factors \(n\). Recently, a counterexample has been shown to a closely related conjectured majorization relation in quantum information theory~\cite{Song2026}. We show a counterexample to~\eqref{eq: maj sus} along similar lines.

\begin{example}
\label{ex: counter}
Consider the following elements in \(\bC^2\):
\begin{align}
    \ket{\psi^{(1)}_2}&=\tfrac{1}{\sqrt{26}}\big(\ket{e_1}+(4-3i)\ket{e_2}\big),\quad &&\ket{\psi^{(1)}_3}=\tfrac{1}{\sqrt{30}}\big((2+4i)\ket{e_1}+(1+3i)\ket{e_2}\big),\nonumber\\
    \ket{\psi^{(2)}_1}&=\tfrac{1}{\sqrt{7}}\big((1+i)\ket{e_1}+(1+2i)\ket{e_2}\big),\quad &&\ket{\psi^{(2)}_3}=\tfrac{1}{\sqrt{28}}\big((5+i)\ket{e_1}+(1+i)\ket{e_2}\big),\nonumber\\
    \ket{\psi^{(3)}_1}&=\tfrac{1}{\sqrt{26}}\big(\ket{e_1}+5\ket{e_2}\big),\quad &&\ket{\psi^{(3)}_2}=\tfrac{1}{\sqrt{30}}\big((2+i)\ket{e_1}+(3+4i)\ket{e_2}\big).\nonumber
\end{align}
We denote the identity on \(\bC^2\) by \(I\). The sum of the \(3\) largest eigenvalues of
\begin{align}
    I \otimes \ket{\psi^{(1)}_2}\bra{\psi^{(1)}_2}\otimes\ket{\psi^{(1)}_3}\bra{\psi^{(1)}_3}+\ket{\psi^{(2)}_1}\bra{\psi^{(2)}_1}\otimes{I} \otimes\ket{\psi^{(2)}_3}\bra{\psi^{(2)}_3}+ \ket{\psi^{(3)}_1}\bra{\psi^{(3)}_1}\otimes\ket{\psi^{(3)}_2}\bra{\psi^{(3)}_2}\otimes I  \nonumber
\end{align}
exceeds the sum of the \(3\) largest coordinates of 
\begin{align}
    \lambda(I) \otimes \lambda(\ket{\psi^{(1)}_2}\bra{\psi^{(1)}_2}) &\otimes \lambda(\ket{\psi^{(1)}_3}\bra{\psi^{(1)}_3}) + \lambda(\ket{\psi^{(2)}_1}\bra{\psi^{(2)}_1}) \otimes \lambda(I) \otimes \lambda(\ket{\psi^{(2)}_3}\bra{\psi^{(2)}_3}) \nonumber \\+ &\lambda (\ket{\psi^{(3)}_1}\bra{\psi^{(3)}_1}) \otimes \lambda(\ket{\psi^{(3)}_2}\bra{\psi^{(3)}_2}) \otimes \lambda(I)  \nonumber
\end{align}
by at least \(0.03\). 
\end{example}

The case where \(n = 2\) and \(m \geq 3\) remains to be addressed.

\subsubsection*{Acknowledgments}
We thank Manny Knill, Noah Lordi, Michael Walter, and Benjamin Schumacher for helpful discussions. MA acknowledges support from the US National Science Foundation Grant PHY-2116246.

\bibliographystyle{plainurl}
\bibliography{references}
\end{document}